\documentclass[11pt]{amsart}

\usepackage{microtype}
\usepackage{amsmath,amssymb,amsthm,mathtools}
\usepackage{enumitem}
\usepackage{booktabs}
\usepackage{array}
\usepackage{float}
\usepackage{graphicx}
\usepackage{xcolor}
\usepackage{stmaryrd}
\definecolor{LEBGray}{gray}{0.45}
\usepackage{tikz}
\usetikzlibrary{arrows.meta}
\usepackage[ruled,vlined,norelsize,noend]{algorithm2e}
\usepackage[hidelinks]{hyperref}
\usepackage[nameinlink,capitalise,noabbrev]{cleveref}

\setlist{itemsep=0.25em,topsep=0.35em}
\numberwithin{equation}{section}
\SetFuncSty{texttt}
\SetKw{Return}{return}
\crefname{algocf}{algorithm}{algorithms}
\Crefname{algocf}{Algorithm}{Algorithms}

\newtheorem{theorem}{Theorem}[section]
\newtheorem{lemma}[theorem]{Lemma}
\newtheorem{proposition}[theorem]{Proposition}
\newtheorem{corollary}[theorem]{Corollary}
\theoremstyle{definition}
\newtheorem{definition}[theorem]{Definition}

\theoremstyle{remark}

\newcommand{\T}{\mathcal T}
\newcommand{\M}{\mathcal M}
\newcommand{\Nnew}{\mathcal N}
\newcommand{\Dsc}{\mathcal D}
\newcommand{\diam}{\operatorname{diam}}
\newcommand{\gen}{\operatorname{gen}}
\newcommand{\Children}{\operatorname{children}}
\newcommand{\Refine}{\textnormal{\texttt{RefineLEB}}}
\newcommand{\RefineOne}{\textnormal{\texttt{RefineOneLEB}}}
\newcommand{\LEPP}{\textnormal{\texttt{BuildLEPP}}}
\newcommand{\card}{\#}
\newcommand{\R}{\mathbb R}

\title[Closure complexity of LEB with terminal-priority]
{Closure complexity of longest-edge bisection for triangular meshes}
\author[Y. Li]{Yuwen Li}
\address{School of Mathematical Sciences, Zhejiang University,
866 Yuhangtang Road, Hangzhou, Zhejiang 310058, People's Republic of China}
\email{liyuwen@zju.edu.cn}
\author[Z. Yang]{Zhiyuan Yang}
\address{School of Mathematical Sciences, Zhejiang University,
866 Yuhangtang Road, Hangzhou, Zhejiang 310058, People's Republic of China}
\email{zhiyuanyang@zju.edu.cn}

\begin{document}

\begin{abstract}
On triangular meshes, we analyze local mesh refinement based on longest-edge bisection equipped with the
serial longest-edge propagation-path closure.  Ties are resolved by terminal priority: if the incoming shared edge is a longest edge of the
neighboring triangle, the pair is declared terminal and that edge is bisected.
For every adaptive mesh sequence $\mathcal{T}_0, \mathcal{T}_1, \ldots, \mathcal{T}_L$ with a sequence of marked sets
$\mathcal{M}_0, \mathcal{M}_1, \ldots, \mathcal{M}_{L-1}$,
we prove the cumulative closure estimate
$$\#\mathcal{T}_L-\#\mathcal{T}_0
  \lesssim\sum_{\ell=0}^{L-1}\#\mathcal{M}_\ell.$$
The proof derives single-mark locality from the uniform multiplicative gap
between possible descendant diameters implied by finite similarity classes,
and converts this locality into the cumulative estimate through a Binev--Dahmen--DeVore-type charging argument.
\end{abstract}

\maketitle

\section{Introduction}\label{sec:introduction}

Adaptive finite element methods (AFEMs) concentrate computational effort by
repeatedly solving a discrete problem, estimating the error, marking selected
elements, and refining the mesh.  The refinement step must reconcile two
competing requirements.  It should remain local enough to preserve the
computational advantage of adaptivity, while at the same time restoring
conformity and maintaining a uniformly nondegenerate family of elements.  In
particular, the elements introduced only for conformity must not overwhelm
those selected by the marking strategy.

Let $\T_0,\ldots,\T_L$ be a sequence of conforming meshes such that, for
$0\le\ell<L$, $\T_{\ell+1}$ is obtained from $\T_\ell$ by applying a fixed
refinement rule to a marked set $\M_\ell\subseteq\T_\ell$.  The quantitative
form of controlled conformity closure is the cumulative estimate
\begin{equation*}
  \card\T_L-\card\T_0
  \le C_{\rm clos}
       \sum_{\ell=0}^{L-1}\card\M_\ell,
  \qquad L\ge1,
\end{equation*}
where the constant $C_{\rm clos}$ is independent of $L$ and of the marked sets, but may
depend on the fixed initial mesh and on the refinement rule.  This is
the Binev--Dahmen--DeVore/Stevenson closure estimate and is one of the standard
complexity inputs in rate-optimal AFEM analysis; see
\cite{BDD04,KPP13,Stevenson08,DGS25}.

Longest-edge bisection (LEB) is a natural geometric refinement rule for planar
triangular meshes.  A triangle is divided by joining the midpoint of one of its
longest edges to the opposite vertex.  Classical results show that repeated
LEB is uniformly angle nondegenerate and produces only finitely many similarity
classes from each initial triangle~\cite{RS75,Stynes80}.  To retain conformity
under local refinement, the  longest-edge propagation-path (LEPP)
procedure follows longest edges through neighboring triangles until it reaches
the boundary or an interior edge that is longest on both sides; the terminal
triangle or terminal pair is then bisected.  This propagation mechanism and
its locality properties have been studied extensively; see
\cite{Rivara84,SPC05,GGR07}.  In three dimensions, longest-edge bisection can
produce degenerating tetrahedra~\cite{Korotov26}.

The closure theory for other bisection rules does not transfer directly to
geometric LEB.  For newest-vertex and Maubach--Traxler bisection, the refinement
edge is encoded by inherited combinatorial data.  Reference edges, vertex
orderings, and generation relations can therefore be propagated through a
refinement tree and compared across neighboring elements.  These structures
underlie the closure analyses in
\cite{BDD04,KPP13,Stevenson08,DGS25}.  In LEB, by contrast, the
bisection edge is recomputed from the current geometry after every refinement.
A Binev--Dahmen--DeVore theorem for LEB was posed as an open problem in
\cite{Gehring23}.

The central difficulty is that the triangle forced by a fine mark may be
substantially coarser than that mark.  A local packing argument centered only
at the current mark then assigns too little charge to the forced refinement.
The proof must also account for earlier marked triangles, without using an
inherited reference edge.

We resolve this difficulty for a precisely specified serial Rivara rule.
The tie convention is part of the algorithm: when the incoming shared edge is
also a longest edge of the neighboring triangle, the path terminates and that
edge is bisected.  We call this convention \emph{terminal priority}.  Under
this convention, the diameter strictly increases at every nonterminal step.
Our main theoretical result is the following.

\begin{theorem}\label{thm:main-intro}
Let $\T_0$ be a finite conforming mesh of a bounded polygonal domain
$\Omega\subset\R^2$. For each $\ell$, call the terminal-priority longest-edge bisection  $\mathcal{T}_{\ell+1}=\Refine(\mathcal{T}_\ell,\mathcal{M}_\ell)$ in \Cref{alg:refine} with marked set
$\mathcal{M}_\ell\subseteq\T_\ell$. Then there is a constant
$C_{\rm LEB}=C_{\rm LEB}(\T_0)$ such that
\begin{equation}\label{eq:main-intro}
  \card\T_L-\card\T_0
  \le C_{\rm LEB}
       \sum_{\ell=0}^{L-1}\card\M_\ell
  \qquad\text{for every }L\ge1.
\end{equation}
\end{theorem}
The finite-similarity-class theorem yields a uniform diameter gap and hence
scale-sensitive locality for the triangles created by one mark.  The cumulative
estimate follows by summing the corresponding bounds over all marks and all new
triangles.

To our knowledge, \Cref{thm:main-intro} is the first cumulative closure
estimate for geometric longest-edge bisection under the stated LEPP rule.  The
constant may depend on $\T_0$, rather than only on its minimum angle, as is
standard in the AFEM framework.  The only imported mathematical result is the
finite-similarity-class theorem; all locality and charging estimates needed
for closure are proved below.

The paper is organized as follows.  \Cref{sec:leb-algorithms} defines the
meshes, generations, and complete serial LEPP algorithm.
\Cref{sec:geometry} develops the diameter spectrum, geometric path growth, and
single-mark locality.  \Cref{sec:charging} proves the cumulative estimate by
the two-sided charging argument.  \Cref{sec:experiments} gives the numerical
comparison on West Lake.

\section{Preliminaries}
\label{sec:leb-algorithms}

\subsection{Conforming meshes and bisection generations}

A conforming mesh $\T$ is a finite family of closed, nondegenerate triangles
with pairwise disjoint interiors such that the intersection of two distinct
triangles is empty, a common vertex, or a common complete edge.  We write
$\overline\Omega:=\bigcup_{K\in\T}K$ and assume that
$\Omega\subset\R^2$ is a bounded polygonal domain.  Thus hanging nodes are
excluded.  We fix an initial mesh $\T_0$.

An edge of $\T$ is called a boundary edge if it is contained in $\partial\Omega$,
and an interior edge otherwise. We write $|K|$ for triangle area, $|e|$ for edge length, $|x|$ for the
Euclidean norm, and $\card A$ for the cardinality of a finite set $A$.

If $K$ has vertices $v_0,v_1,v_2$ and $e=[v_i,v_j]$ is a longest edge, the
longest-edge bisection of $K$ connects the midpoint of $e$ to the opposite vertex.
More precisely, after relabeling so that $K=[v_0,v_1,v_2]$ and
$e=[v_0,v_1]$, with $m:=\operatorname{mid}(e)$, write
\[
  \Children(K,e):=\bigl\{[v_0,m,v_2],[m,v_1,v_2]\bigr\}.
\]
The two children have equal area $|K|/2$.

\begin{figure}[!htbp]
\centering
\begin{tikzpicture}[scale=0.92,line cap=round,line join=round,>=Latex]
  \begin{scope}
    \coordinate (v0) at (0,0);
    \coordinate (v1) at (3.4,0);
    \coordinate (v2) at (0.85,2.1);
    \coordinate (m) at (1.7,0);
    \draw[thick] (v0)--(v1)--(v2)--cycle;
    \draw[ultra thick] (v0)--(v1);
    \fill (m) circle (1.8pt);
    \node[below left] at (v0) {$v_0$};
    \node[below right] at (v1) {$v_1$};
    \node[above] at (v2) {$v_2$};
    \node[below] at (m) {$m$};
    \node at (1.25,0.82) {$K$};
    \node at (1.7,-0.65) {longest edge $e$};
  \end{scope}
  \draw[->,very thick] (4.05,1.0)--(5.15,1.0)
    node[midway,above,font=\small] {bisect};
  \begin{scope}[xshift=5.8cm]
    \coordinate (w0) at (0,0);
    \coordinate (w1) at (3.4,0);
    \coordinate (w2) at (0.85,2.1);
    \coordinate (n) at (1.7,0);
    \draw[thick] (w0)--(w1)--(w2)--cycle;
    \draw[ultra thick] (w0)--(n);
    \draw[ultra thick] (n)--(w1);
    \draw[thick,densely dashed,LEBGray] (n)--(w2);
    \fill (n) circle (1.8pt);
    \node[below left] at (w0) {$v_0$};
    \node[below right] at (w1) {$v_1$};
    \node[above] at (w2) {$v_2$};
    \node[below] at (n) {$m$};
    \node at (0.78,0.72) {$K_1$};
    \node at (2.05,0.72) {$K_2$};
  \end{scope}
\end{tikzpicture}
\caption{Longest-edge bisection of $K$.  The selected edge $e=[v_0,v_1]$
is drawn thick; the dashed segment joins its midpoint $m$ to the opposite
vertex.  The two children are $K_1=[v_0,m,v_2]$ and $K_2=[m,v_1,v_2]$.}
\label{fig:leb-bisection}
\end{figure}
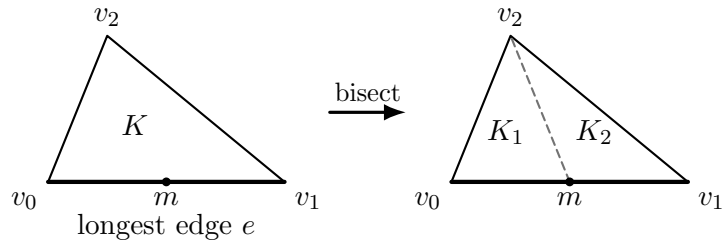

For each triangle $K$ obtained from $\T_0$ by successive bisections, let
$R(K)\in\T_0$ be the unique initial triangle containing $K$.  The generation
$\gen(K)$ is the number of successive bisections from $R(K)$ to $K$.  Thus
\[
  \gen(K)=0\quad(K\in\T_0),
  \qquad
  \gen(K')=\gen(K)+1
\]
whenever $K'$ is a child of $K$.  Hence
\begin{equation}\label{eq:area-generation}
  |K|=2^{-\gen(K)}|R(K)|.
\end{equation}
We write $\diam(K)$ for the longest-edge length and $z_K$ for the
barycenter of $K$.

\subsection{Propagation paths and terminal priority}

Let $\T$ be a conforming mesh.  For an interior edge $e$ of
$K\in\T$, let $K^*$ denote the unique triangle in $\T\setminus\{K\}$
sharing $e$.

\begin{definition}[Terminal-priority LEPP and terminal edge]\label{def:lepp}
Starting from $K_0\in\T$, construct a sequence recursively as follows.  At
$K_j$, choose any longest edge $e_j$ of $K_j$.
\begin{enumerate}[label=(\roman*)]
\item If $e_j\subset\partial\Omega$, stop.  Then $K_j$ is the terminal
      triangle, $e_j$ is the terminal edge, and the terminal set is
      $\mathcal S=\{K_j\}$.
\item Otherwise, if $e_j$ is also a longest edge of $K_j^*$, stop.
      Then $(K_j,K_j^*)$ is the
      terminal pair, $e_j$ is the terminal edge, and the terminal set is
      $\mathcal S=\{K_j,K_j^*\}$.
\item If $e_j$ is not a longest edge of $K_j^*$, set $K_{j+1}:=K_j^*$ and
      continue.
\end{enumerate}
The resulting sequence $(K_0,\ldots,K_J)$ is called a longest-edge
propagation path (LEPP) with terminal priority.  In case~(iii),
$K_j\to K_{j+1}$ is called a nonterminal step.
\end{definition}

This is the Rivara--LEPP convention; see \cite{Rivara84,Rivara97,SPC05}.

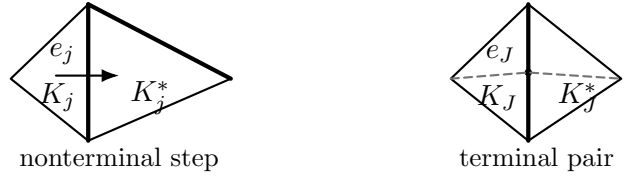
\begin{figure}[!htbp]
\centering
\begin{tikzpicture}[scale=0.82,line cap=round,line join=round,>=Latex]
  \begin{scope}
    \coordinate (a) at (0,0);
    \coordinate (u) at (1.25,1.2);
    \coordinate (d) at (1.25,-1.0);
    \coordinate (b) at (3.55,0);
    \draw[thick] (a)--(u)--(b)--(d)--cycle;
    \draw[ultra thick] (u)--(d);
    \draw[ultra thick] (u)--(b);
    \draw[->,thick] (0.73,0.05)--(1.75,0.05);
    \node at (0.75,-0.30) {$K_j$};
    \node at (2.25,-0.30) {$K_j^*$};
    \node[left] at (1.20,0.40) {$e_j$};
    \node[align=center,font=\small,text width=5.4cm] at (1.75,-1.30)
      {nonterminal step};
  \end{scope}
  \begin{scope}[xshift=7.1cm]
    \coordinate (a) at (0,0);
    \coordinate (u) at (1.25,1.2);
    \coordinate (d) at (1.25,-1.0);
    \coordinate (b) at (2.75,0);
    \coordinate (m) at (1.25,0.1);
    \draw[thick] (a)--(u)--(b)--(d)--cycle;
    \draw[ultra thick] (u)--(d);
    \fill (m) circle (1.7pt);
    \draw[thick,densely dashed,LEBGray] (a)--(m)--(b);
    \node at (0.75,-0.30) {$K_J$};
    \node at (2.05,-0.30) {$K_J^*$};
    \node[left] at (1.20,0.40) {$e_J$};
    \node[align=center,font=\small,text width=5.2cm] at (1.38,-1.30)
      {terminal pair};
  \end{scope}
\end{tikzpicture}
\caption{A nonterminal step (left) and a terminal pair (right).
Thick edges are longest in their triangles; dashed segments are new edges.}
\label{fig:terminal-priority}
\end{figure}

The following algorithm implements Definition~\ref{def:lepp} and returns
the path, terminal edge, and terminal set.

\begin{algorithm}[!htbp]
\caption{Build a terminal-priority LEPP \LEPP}
\label{alg:lepp}
\KwIn{A conforming mesh $\T$ and a starting triangle $K_0\in\T$}
\KwOut{a LEPP $\mathcal P$, its terminal edge $e$, and its terminal set
$\mathcal S$}
$K\gets K_0$ and $\mathcal P\gets(K_0)$\;
\While{true}{
  Choose any longest edge $e$ of $K$\;
  \If{$e\subset\partial\Omega$}{
    $\mathcal S\gets\{K\}$\;
    \Return{$(\mathcal P,e,\mathcal S)$}\;
  }
  Find the neighbor $K^*$\;
  \If{$e$ is a longest edge of $K^*$}{
    $\mathcal S\gets\{K,K^*\}$\;
    \Return{$(\mathcal P,e,\mathcal S)$}\;
  }
  $K\gets K^*$ and append $K$ to $\mathcal P$\;
}
\end{algorithm}

At the second \texttt{if} statement in \Cref{alg:lepp}, the selected edge $e$
is accepted whenever it ties for longest in $K^*$.  Thus the diameter strictly
increases at every nonterminal step.  The path cannot revisit a triangle,
so the loop terminates because the current mesh is finite.

A triangle $M\in\T$ selected for refinement is called a marked triangle.
A subset $\M\subseteq\T$ of marked triangles is called a marked set.  In a
call $\RefineOne(\T,M)$, the triangle $M$ is the mark of the call.

\begin{algorithm}[!htbp]
\caption{Single-mark refinement \RefineOne}
\label{alg:refine-one}
\KwIn{A conforming mesh $\T$ and a marked triangle $M\in\T$}
\KwOut{upon termination, the conforming mesh $\RefineOne(\T,M)$ in which
$M$ has been bisected}
$\widehat{\T}\gets\T$\;
\While{$M\in\widehat{\T}$}{
  $(\mathcal P,e,\mathcal S)\gets\LEPP(\widehat{\T},M)$\;
  $\mathcal C\gets\displaystyle\bigcup_{K\in\mathcal S}\Children(K,e)$\;
  $\widehat{\T}\gets(\widehat{\T}\setminus\mathcal S)\cup\mathcal C$\;
}
\Return{$\widehat{\T}$}\;
\end{algorithm}

In \Cref{alg:refine-one}, a boundary terminal edge is split in its single
adjacent triangle; an interior terminal edge is split simultaneously in both
adjacent triangles.  No other old
edge is subdivided, so every intermediate mesh is conforming.  An iteration
either bisects $M$ or leaves it unchanged.  Every triangle bisected during the
call thus belongs to the terminal set of a current LEPP starting at $M$.

\begin{algorithm}[!htbp]
\caption{Marked-set refinement \Refine}
\label{alg:refine}
\KwIn{A conforming mesh $\T$ and a marked set
$\M=\{M_1,\ldots,M_r\}\subseteq\T$ in a fixed serial order}
\KwOut{the conforming mesh $\T^+=\Refine(\T,\M)$}
$\T^+\gets\T$\;
\For{$i\gets1$ \KwTo $r$}{
  \If{$M_i\in\T^+$}{
    $\T^+\gets\RefineOne(\T^+,M_i)$\;
  }
}
\Return{$\T^+$}\;
\end{algorithm}

We call the refinement rule of \Cref{alg:lepp,alg:refine-one,alg:refine}
terminal-priority longest-edge bisection (TP-LEB).
A marked triangle already bisected by an earlier call is skipped in
\Cref{alg:refine}.  The estimates are uniform over all longest-edge choices
permitted by \Cref{alg:lepp} and all fixed serial orders of the marked set in
\Cref{alg:refine}.  Termination of \RefineOne{} is established in
\Cref{cor:refine-one-termination}.

\section{Geometric structure and single-mark locality}
\label{sec:geometry}

\subsection{Finite similarity classes and diameter scales}

For a nondegenerate triangle $K$, let $\Dsc(K)$ denote $K$ and all triangles
obtainable from $K$ by successive longest-edge bisections, including both
children and all tied-longest choices.
A triangle is a possible descendant of $\T_0$ if it belongs to $\Dsc(T)$
for some $T\in\T_0$.

\begin{theorem}[Finite similarity classes]\label{thm:FSC}
For every nondegenerate triangle $K$, the set $\Dsc(K)$ contains
only finitely many similarity classes.
\end{theorem}

This theorem is due to Stynes~\cite{Stynes80}; see also
\cite{GGR07,KS26}.

For a triangle $K$, the quantity $\diam(K)/|K|^{1/2}$ is invariant under
similarity.  Combining Theorem~\ref{thm:FSC} with
\eqref{eq:area-generation} gives the following lemma.

\begin{lemma}[Finite $\sqrt2$-lattice spectrum]\label{lem:spectrum}
There is a finite set $\Sigma=\Sigma(\T_0)\subset(0,\infty)$ such that every
possible descendant $K$ of $\T_0$ satisfies
\begin{equation*}
  \diam(K)=\sigma2^{-\gen(K)/2}
  \qquad\text{for some }\sigma\in\Sigma.
\end{equation*}
\end{lemma}

\begin{proof}
Fix $T\in\T_0$.  Choose one representative $S$ from each of the finitely many
similarity classes in $\Dsc(T)$.  If $K$ lies in the class of $S$, similarity
and \eqref{eq:area-generation} give
\[
  \diam(K)
  =\frac{\diam(S)}{|S|^{1/2}}|K|^{1/2}
  =\left(\frac{\diam(S)}{|S|^{1/2}}|T|^{1/2}\right)
   2^{-\gen(K)/2}.
\]
Collect the parenthesized constants over the finitely many classes and the
finitely many initial triangles $T\in\T_0$.
\end{proof}

Fix such a set $\Sigma$, and retain the notation
$a=\min\Sigma$ and $b=\max\Sigma$.  Then every possible descendant $K$ of
$\T_0$ satisfies
\begin{equation}\label{eq:diam-gen}
  a2^{-\gen(K)/2}\le \diam(K)\le b2^{-\gen(K)/2}.
\end{equation}

The essential consequence is stronger than shape regularity: distinct possible
diameters cannot be arbitrarily close multiplicatively.

\begin{lemma}[Uniform diameter gap]\label{lem:rho}
There exists $\rho=\rho(\T_0)>1$ such that, for any two possible descendant
diameters $h_1<h_2$,
\begin{equation*}
  h_2\ge\rho h_1.
\end{equation*}
\end{lemma}

\begin{proof}
Set $k:=h_2/h_1$.  By Lemma~\ref{lem:spectrum},
\[
  k=\frac{\tau}{\sigma}2^{m/2}>1
  \qquad\text{for some }\sigma,\tau\in\Sigma,\ m\in\mathbb Z.
\]
For fixed $(\sigma,\tau)$, the constraint $(\tau/\sigma)2^{m/2}>1$
is equivalent to $m>2\log_2(\sigma/\tau)$ and hence has a least integer
solution.  Since $(\tau/\sigma)2^{m/2}$ is strictly increasing in $m$,
\[
  \rho_{\sigma,\tau}:=
  \min_{\substack{m\in\mathbb Z\\(\tau/\sigma)2^{m/2}>1}}
  \frac{\tau}{\sigma}2^{m/2}>1.
\]
Finally, $\rho:=\min_{\sigma,\tau\in\Sigma}\rho_{\sigma,\tau}>1$,
since $\Sigma$ is finite.  Thus $k\ge\rho$, which yields $h_2\ge\rho h_1$.
\end{proof}

\subsection{Geometric growth along a LEPP}

In this subsection, LEPPs are taken on conforming meshes of longest-edge
descendants of $\T_0$.

\begin{lemma}[Strict geometric growth]\label{lem:path-growth}
Let $(K_0,\ldots,K_J)$ be a LEPP as in
Definition~\ref{def:lepp}.  Then
\begin{equation*}
  \diam(K_{j+1})\ge\rho\diam(K_j),
  \qquad 0\le j<J,
\end{equation*}
where $\rho>1$ is from Lemma~\ref{lem:rho}.  Hence
\begin{equation*}
  \sum_{j=0}^{J}\diam(K_j)
  \le \frac{1}{1-\rho^{-1}}\diam(K_J).
\end{equation*}
\end{lemma}

\begin{proof}
The edge $e_j$ is longest in $K_j$ but not in $K_{j+1}$, so
$\diam(K_j)=|e_j|<\diam(K_{j+1})$.
Lemma~\ref{lem:rho} gives the growth bound.  Summing the resulting
geometric series backwards from $\diam(K_J)$ proves the second bound.
\end{proof}

\begin{corollary}[Reach of a propagation path]\label{cor:path-reach}
There is a constant $C_{\rm path}=C_{\rm path}(\T_0)>0$ such that, for every
LEPP starting at $M$ with terminal set $\mathcal S$ and every $P\in\mathcal S$,
\begin{equation*}
  \diam(P)\ge\diam(M),
  \qquad
  |z_P-z_M|\le C_{\rm path}\diam(P).
\end{equation*}
\end{corollary}

\begin{proof}
Let $(K_0,\ldots,K_J)$ be the LEPP, with $K_0=M$.
Adjacent triangles $K,K'$ satisfy
\[
  |z_K-z_{K'}|\le\diam(K)+\diam(K').
\]
By Definition~\ref{def:lepp}, $\diam(P)=\diam(K_J)$ and
$|z_P-z_{K_J}|\le2\diam(P)$.
Lemma~\ref{lem:path-growth} yields $\diam(P)\ge\diam(M)$ and
\begin{align*}
  |z_P-z_M|
  &\le 2\diam(P)+2\sum_{j=0}^{J}\diam(K_j)\\
  &\le \left(2+\frac{2}{1-\rho^{-1}}\right)\diam(P).
\end{align*}
Hence one may take $C_{\rm path}:=2+2/(1-\rho^{-1})$.
\end{proof}

\subsection{Generation and spatial locality for one mark}

In this subsection, let $\T$ be any conforming mesh obtainable from $\T_0$
by finitely many longest-edge bisections, and let $M\in\T$.
We first estimate the triangles created by $\RefineOne(\T,M)$,
without assuming that the call terminates.

\begin{lemma}[Single-mark locality]
\label{lem:single-mark}
There are constants $\Gamma\in\mathbb N$ and $D>0$, depending only on
$\T_0$, such that, in every iteration of the while loop in
\Cref{alg:refine-one} with input $(\T,M)$, every triangle
$S\in\mathcal C$ satisfies
\begin{align*}
  \gen(S)&\le\gen(M)+\Gamma,\\
  |z_S-z_M|&\le D\,2^{-\gen(S)/2}.
\end{align*}
\end{lemma}

\begin{proof}
Fix an iteration and $S\in\mathcal C$.  Let $P\in\mathcal S$ satisfy
$S\in\Children(P,e)$.  By \Cref{alg:refine-one}, $P$ belongs to the terminal
set of the current LEPP starting at $M$.

Corollary~\ref{cor:path-reach} gives $\diam(P)\ge\diam(M)$.  Using
\eqref{eq:diam-gen},
\[
  b2^{-\gen(P)/2}
  \ge\diam(P)
  \ge\diam(M)
  \ge a2^{-\gen(M)/2}.
\]
Consequently,
\[
  \gen(P)
  \le \gen(M)+2\log_2(b/a).
\]
Since $\gen(S)=\gen(P)+1$, the generation bound follows with
\begin{equation*}
  \Gamma:=1+\left\lceil2\log_2(b/a)\right\rceil.
\end{equation*}

For the distance bound, the barycenter of a child lies in its parent, so
$|z_S-z_P|\le\diam(P)$.  Corollary~\ref{cor:path-reach} and
\eqref{eq:diam-gen} give
\[
  |z_S-z_M|
  \le (1+C_{\rm path})\diam(P)
  \le \sqrt2(1+C_{\rm path})b\,2^{-\gen(S)/2}.
\]
Thus the distance bound holds with
$D:=\sqrt2(1+C_{\rm path})b$.
\end{proof}

\begin{corollary}[Termination]
\label{cor:refine-one-termination}
The call $\RefineOne(\T,M)$ terminates and returns a
conforming refinement of $\T$ in which $M$ has been bisected.
\end{corollary}

\begin{proof}
Each call to \LEPP{} terminates on the current finite mesh by strict
diameter growth.  By Lemma~\ref{lem:single-mark}, every triangle bisected
in an iteration has generation at most $\gen(M)+\Gamma-1$.
Since the initial mesh is finite and each bisection produces two children,
only finitely many triangles of these generations can occur in any execution.
Every iteration bisects at least one previously unbisected triangle,
so the loop terminates.  Its
stopping condition implies that $M$ has been bisected, and conformity follows
from the simultaneous replacement in \Cref{alg:refine-one}.
\end{proof}

\section{The cumulative charging argument}
\label{sec:charging}

We first introduce the charging function.  Consider the refinement sequence
\begin{equation*}
  \T_{\ell+1}=\Refine(\T_\ell,\M_\ell),
  \qquad \M_\ell\subseteq\T_\ell,
  \qquad 0\le\ell<L.
\end{equation*}
Let $\mathfrak M$ be the set of marks of all calls to \RefineOne{} in this
sequence.  Since each call bisects and thereby removes its mark from the
current mesh, no triangle is the mark of two distinct calls.  In particular,
\begin{equation*}
  \card\mathfrak M
  \le\sum_{\ell=0}^{L-1}\card\M_\ell.
\end{equation*}
Set $\Nnew:=\T_L\setminus\T_0$.  Let $\Gamma$ and $D$ be as in
Lemma~\ref{lem:single-mark}, and put $E:=2D$.

\begin{definition}[Charging function]\label{def:lambda}
For $T\in\Nnew$ and $M\in\mathfrak M$, define the charge from $M$ to $T$ by
\begin{equation*}
\lambda(T,M):=
\begin{cases}
2^{(\gen(T)-\gen(M))/2},
&\begin{aligned}
  &\gen(T)\le\gen(M)+\Gamma,\\[-0.1em]
  &|z_T-z_M|\le E2^{-\gen(T)/2},
\end{aligned}\\[0.75em]
0,&\text{otherwise.}
\end{cases}
\end{equation*}
\end{definition}

We first bound
$\sum_{T\in\Nnew}\lambda(T,M)$ uniformly in $M\in\mathfrak M$.  Set
\begin{equation*}
  A_{\min}:=\min_{T\in\T_0}|T|>0.
\end{equation*}

\begin{lemma}[Packing at one generation]\label{lem:packing}
There is $C_{\rm pack}=C_{\rm pack}(\T_0)<\infty$ such that, for every
point $x\in\R^2$ and every integer $g\ge0$, at most $C_{\rm pack}$ triangles
$T\in\T_L$ can satisfy simultaneously
\[
  \gen(T)=g,
  \qquad |z_T-x|\le E2^{-g/2}.
\]
\end{lemma}

\begin{proof}
By \eqref{eq:diam-gen}, every such $T$ is contained in the ball
\[
  B\bigl(x,(E+b)2^{-g/2}\bigr).
\]
By \eqref{eq:area-generation},
$|T|=2^{-g}|R(T)|\ge A_{\min}2^{-g}$.  These triangles have pairwise disjoint
interiors.  Comparing their total area with the area of the containing ball
gives
\[
  \card\bigl\{T\in\T_L:\gen(T)=g,\ |z_T-x|\le E2^{-g/2}\bigr\}
  \le \frac{\pi(E+b)^2}{A_{\min}}.
\]
\end{proof}

\begin{proposition}[Bounded charge from one mark]\label{prop:upper-dollar}
For every $M\in\mathfrak M$,
\begin{equation*}
  \sum_{T\in\Nnew}\lambda(T,M)
  \le C_\uparrow,
\end{equation*}
where
\begin{equation*}
  C_\uparrow
  :=C_{\rm pack}
    \frac{2^{(\Gamma+1)/2}}{\sqrt2-1}
\end{equation*}
depends only on $\T_0$ and the rule.
\end{proposition}

\begin{proof}
Write $m=\gen(M)$.  If $\lambda(T,M)\ne0$, then
$g:=\gen(T)\le m+\Gamma$.  Lemma~\ref{lem:packing}, with $x=z_M$, gives
at most $C_{\rm pack}$ such triangles at every fixed
generation $g$.  Hence
\begin{align*}
  \sum_{T\in\Nnew}\lambda(T,M)
  &\le C_{\rm pack}
       \sum_{g=0}^{m+\Gamma}2^{(g-m)/2}\\
  &\le C_{\rm pack}
       \frac{2^{(\Gamma+1)/2}}{\sqrt2-1}.
\end{align*}
\end{proof}

\begin{proposition}[Unit charge]
\label{prop:lower-dollar}
For every $T\in\Nnew$,
\begin{equation*}
  \sum_{M\in\mathfrak M}\lambda(T,M)\ge1.
\end{equation*}
\end{proposition}

\begin{proof}
Fix $T\in\Nnew$ and set $K_0:=T$.  Whenever $K_{r-1}\notin\T_0$,
let $K_r$ be the mark of the unique call to \RefineOne{}
that created $K_{r-1}$.  Each mark exists before its call and, unless it
belongs to $\T_0$, was created during an earlier call.  Thus the construction ends at
$K_q\in\T_0$, and $K_1,\ldots,K_q$ are distinct members of $\mathfrak M$.
By Lemma~\ref{lem:single-mark}, for $1\le r\le q$,
\begin{align*}
  \gen(K_{r-1})&\le\gen(K_r)+\Gamma,\\
  |z_{K_{r-1}}-z_{K_r}|&\le
  D2^{-\gen(K_{r-1})/2}.
\end{align*}
Put $t:=\gen(T)$ and note that $t\ge1$ whereas
$\gen(K_q)=0$.  Let
\[
  s:=\min\{1\le r\le q:\gen(K_r)<t\}.
\]
Then $t\le\gen(K_{s-1})\le\gen(K_s)+\Gamma$, so
\begin{equation*}
  t-\Gamma\le\gen(K_s)<t.
\end{equation*}

Set $B:=\overline B(z_T,E2^{-t/2})$.  Suppose first that
\[
  |z_T-z_{K_r}|\le E2^{-t/2}
  \qquad(1\le r\le s).
\]
Since $t\le\gen(K_s)+\Gamma$, Definition~\ref{def:lambda} gives
\[
  \lambda(T,K_s)
  =2^{(t-\gen(K_s))/2}\ge1.
\]

Otherwise, let $j\in\{1,\ldots,s\}$ be the first index for which
\begin{equation*}
  |z_T-z_{K_j}|>E2^{-t/2}.
\end{equation*}
Since $|z_T-z_{K_1}|\le D2^{-t/2}<E2^{-t/2}$, necessarily $j\ge2$.
For $1\le r<j$, the barycenter $z_{K_r}$ lies in the ball $B$, and
$\gen(K_r)\ge t$ by the minimality of $s$.  Thus
$\lambda(T,K_r)=2^{(t-\gen(K_r))/2}$.  The triangle inequality and
Lemma~\ref{lem:single-mark} give
\begin{align*}
  E2^{-t/2}
  &<|z_{K_0}-z_{K_j}|\\
  &\le D\sum_{r=1}^{j}2^{-\gen(K_{r-1})/2}\\
  &=D2^{-t/2}
    \left(1+\sum_{r=1}^{j-1}
    2^{(t-\gen(K_r))/2}\right).
\end{align*}
Since $E=2D$, it follows that
\[
  \sum_{r=1}^{j-1}\lambda(T,K_r)>1.
\]
This proves the proposition in both cases, illustrated in
\Cref{fig:charge-cases}.
\end{proof}

\begin{figure}[t]
\centering
\begin{tikzpicture}[scale=0.95,>=Latex,line cap=round,
  every node/.style={font=\small}]
  \begin{scope}
    \draw[thick,fill=black!3] (0,0) circle (1.65);
    \draw[->,thick] (0,0)--(0.45,0.42);
    \draw[->,thick] (0.45,0.42)--(1.04,0.75);
    \fill (0,0) circle (1.6pt) node[below left] {$z_T$};
    \fill (0.45,0.42) circle (1.6pt) node[above left] {$\cdots$};
    \fill (1.04,0.75) circle (1.6pt) node[above left] {$z_{K_s}$};
    \draw[densely dotted] (0,0)--(0,-1.65)
      node[midway,right] {$E2^{-t/2}$};
    \node[align=center] at (0,-2.25)
      {$\gen(K_s)<t$\\$z_{K_1},\ldots,z_{K_s}\in B$};
  \end{scope}
  \begin{scope}[xshift=5.8cm]
    \draw[thick,fill=black!3] (0,0) circle (1.65);
    \draw[->,thick] (0,0)--(0.45,0.25);
    \draw[->,thick] (0.45,0.25)--(1.10,0.60);
    \draw[->,thick] (1.10,0.60)--(1.76,0.86);
    \fill (0,0) circle (1.6pt) node[below left] {$z_T$};
    \fill (0.45,0.25) circle (1.6pt) node[above left] {$\cdots$};
    \fill (1.10,0.60) circle (1.6pt);
    \node at (1.10,0.10) {$z_{K_{j-1}}$};
    \fill (1.76,0.86) circle (1.6pt) node[above] {$z_{K_j}$};
    \node[align=center] at (0.25,-2.25)
      {$j\le s$\\$z_{K_1},\ldots,z_{K_{j-1}}\in B,\quad z_{K_j}\notin B$};
  \end{scope}
\end{tikzpicture}
\caption{The two cases in \Cref{prop:lower-dollar}.
The sequence of barycenters stays in $B$ through $z_{K_s}$ (left), or first
leaves $B$ at $z_{K_j}$, $j\le s$ (right).}
\label{fig:charge-cases}
\end{figure}
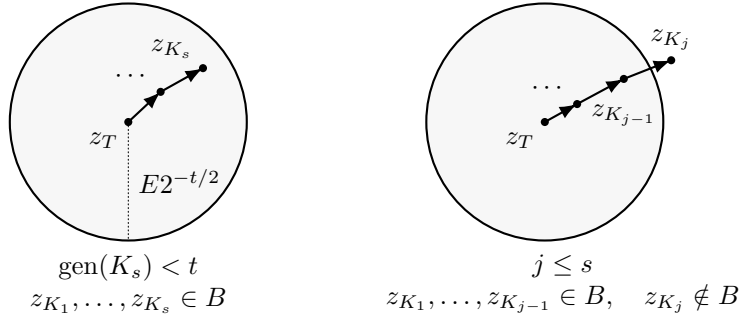

\begin{proof}[Proof of Theorem~\ref{thm:main-intro}]
Propositions~\ref{prop:upper-dollar} and \ref{prop:lower-dollar} give
\begin{align*}
  \card\Nnew
  &\le \sum_{T\in\Nnew}
       \sum_{M\in\mathfrak M}\lambda(T,M)\\
  &=\sum_{M\in\mathfrak M}
       \sum_{T\in\Nnew}\lambda(T,M)\\
  &\le C_\uparrow\card\mathfrak M\\
  &\le C_\uparrow
       \sum_{\ell=0}^{L-1}\card\M_\ell.
\end{align*}

Since $\Nnew=\T_L\setminus\T_0$,
\[
  \card\T_L-\card\T_0
  \le\card\Nnew.
\]
Together, these bounds prove \eqref{eq:main-intro} with
$C_{\rm LEB}:=C_\uparrow$.
\end{proof}

\section{Numerical experiments}
\label{sec:experiments}

We compare TP-LEB with newest-vertex bisection (NVB) and red--green
refinement (RG) on a West Lake domain.  The comparison concerns cumulative
mesh growth, residual-estimator reduction, and mesh shape quality under a
common bound on the number of elements.

\subsection{Experimental setting}

The test domain is the main connected water component of West Lake in
Hangzhou, taken from the OpenStreetMap water
multipolygon~\cite{OpenStreetMap}.  Its island and causeway boundaries
are retained; a separate pond within an island is omitted.
The initial mesh has $2\,333$ triangles and is generated by the
Frontal--Delaunay algorithm of Gmsh~\cite{Gmsh09}.

On this domain we solve
\begin{equation*}
  -\Delta u=1\quad\text{in }\Omega,
  \qquad u=0\quad\text{on }\partial\Omega,
\end{equation*}
using continuous piecewise affine finite elements.  The Dirichlet
condition is imposed on every boundary component.
The experiments are run in MATLAB R2026a, with sparse direct solution of
the discrete systems.

For the Galerkin solution $u_\T$, define the residual estimator by
\begin{equation*}
\begin{aligned}
  \eta_K^2
  &:=\diam(K)^2|K|+
     \frac12\sum_{e\in\mathcal E_\T^{\rm int}(K)}
     |e|\bigl\|\llbracket\nabla u_\T\cdot n_e\rrbracket\bigr\|_{L^2(e)}^2,
     \qquad K\in\T,\\
  \eta_\T&:=\left(\sum_{K\in\T}\eta_K^2\right)^{1/2},
\end{aligned}
\end{equation*}
where $\mathcal E_\T^{\rm int}(K)$ is the set of interior edges of $K$.
Here $n_e$ is a fixed unit normal on $e$, and the bracket $\llbracket\cdot\rrbracket$ denotes the
jump across that edge.  We report estimator reduction without using
an exact solution.

All methods start from the same initial mesh.
For each method, $\T_\ell$ is the mesh after $\ell$ adaptive refinements,
and refinement stops at the largest level $L$ satisfying
$\card\T_L\le 40\,000$.
For $0\le\ell<L$, the marked set $\M_\ell\subseteq\T_\ell$ is obtained
by taking the largest indicators until
\[
  \sum_{K\in\M_\ell}\eta_K^2
  \ge0.35\sum_{K\in\T_\ell}\eta_K^2.
\]
Refining $\M_\ell$ gives $\T_{\ell+1}$.
Each method uses its own current solution and indicators.
TP-LEB follows \Cref{alg:lepp,alg:refine-one,alg:refine}; NVB is initialized
with a longest reference edge on every triangle~\cite{KPP13,Chen09}.
RG uses temporary green completion: a marked green triangle triggers
replacement of its green patch by red refinement of the parent
triangle~\cite{Mitchell17}.
For TP-LEB, a squared edge length $s$ is treated as tied with the largest
squared edge length $s_{\max}$ whenever
$s_{\max}-s\le10^{-12}\max\{1,s_{\max}\}$.

Write $N_\ell=\card\T_\ell$ and $\eta_\ell=\eta_{\T_\ell}$.  We measure cumulative
element growth per marked triangle by
\begin{equation*}
  G_\ell:=\frac{N_\ell-N_0}{\sum_{j=0}^{\ell-1}\card\M_j},
  \qquad \ell=1,\ldots,L.
\end{equation*}
For RG, the numerator includes the net change after green-patch replacement;
its local refinement patterns differ from binary bisection.
We also record the normalized shape parameter
$\gamma(\T_L)/\gamma(\T_0)$, where
\[
  \gamma(\T):=\max_{K\in\T}\frac{\diam(K)}{2r_K},
\]
and $r_K$ is the inradius of $K$.

\begingroup
\setlength{\intextsep}{8pt}
\begin{figure}[!htbp]
\centering
\includegraphics[width=0.9\textwidth]{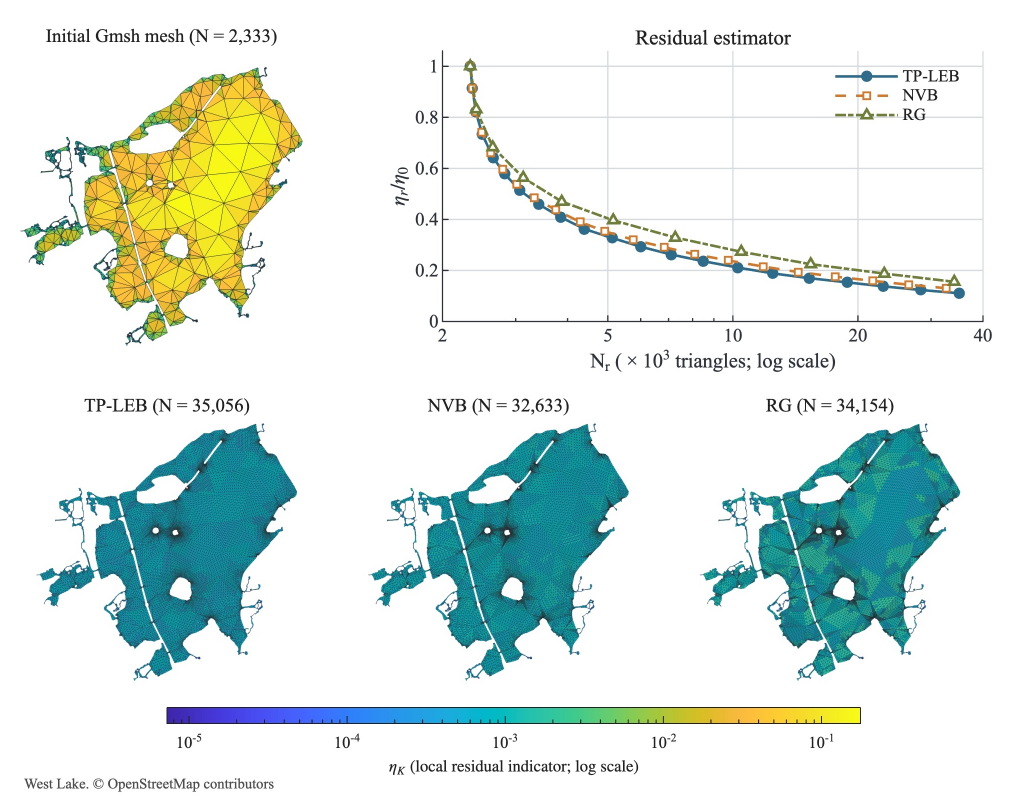}
\caption{West Lake comparison.  Top: initial mesh and estimator reduction;
bottom: final indicator distributions.  Data: \copyright\ OpenStreetMap
contributors (ODbL)~\cite{OpenStreetMap}.}
\label{fig:west-lake-comparison}
\end{figure}
\endgroup

\subsection{Estimator reduction and shape quality}

\Cref{tab:numerical-summary} collects the results at termination.
The cumulative growth ratios for TP-LEB and NVB are
close, whereas RG has a larger ratio.  TP-LEB has the smallest final
shape ratio in this test.

\begin{table}[!htbp]
\centering
\small
\setlength{\tabcolsep}{4.5pt}
\caption{West Lake results under the $40\,000$-element bound.}
\label{tab:numerical-summary}
\begin{tabular}{lrrrrr}
\toprule
Method & $L$ & $N_L$ & $G_L$ & $\eta_L/\eta_0$ & $\gamma(\T_L)/\gamma(\T_0)$ \\
\midrule
TP-LEB & 20 & 35,056 & 1.354 & 0.111 & 1.897 \\
NVB     & 20 & 32,633 & 1.406 & 0.130 & 1.942 \\
RG      & 10 & 34,154 & 4.852 & 0.155 & 1.994 \\
\bottomrule
\end{tabular}
\end{table}

\Cref{fig:west-lake-comparison} shows the estimator
ratio $\eta_\ell/\eta_0$ against $N_\ell$ for $\ell=0,\ldots,L$.
In the later part of the sampled mesh range, TP-LEB gives smaller residual
estimators than NVB and RG at comparable element counts.
The final values correspond to different meshes; the curves provide
the comparison over their common element range.

\end{document}